\documentclass[12pt,reqno]{amsart}

\usepackage{times}
\usepackage{amsfonts}
\usepackage{amssymb}
\usepackage[colorlinks, linkcolor=blue, citecolor=blue]{hyperref}
\usepackage{cleveref}
\usepackage{geometry}

\usepackage{setspace}

\newtheorem{theorem}{Theorem}[section]
\newtheorem{corollary}[theorem]{Corollary}
\newtheorem{lemma}[theorem]{Lemma}

\theoremstyle{definition}

\theoremstyle{remark}

\allowdisplaybreaks
\numberwithin{equation}{section}
\begin{document}
	
	\title[Difference of weighted composition operators]
	{Essential norm and Schatten class difference of weighted composition operators over the ball}
	
	\author[X. Hu]{Xiaohe Hu}
	\address[X. Hu]{School of Mathematics and Statistics, Henan Normal University, Xinxiang 453007, China.}
	\email{huxiaohe@htu.edu.cn}

	\author[Z. Yang]{Zicong Yang}
	\address[Z. Yang]{Institute of Mathematics, Hebei University of Technology, Tianjin 300401, China.}
	\email{zc25@hebut.edu.cn; zicongyang@126.com}

	\thanks{Hu was supported in part by Natural Science Foundation of Henan Province (Grant No. 252300421778). Yang was supported in part by the Natural Science Foundation of Hebei Province (Grant No. A2020202031, A2020202037).}

\begin{abstract}
In this paper, we obtain the essential norm estimate for the difference of two weighted composition operators acting on standard weighted Bergman spaces over the unit ball. And we get some characterizations for the difference of weighted composition operators belonging to Schatten class, which has rarely been considered before. Our methods are fundamental, which involve Carleson measures and $R$-Berezin transform.
\end{abstract}
\maketitle

\section{Introduction}

Let $\mathbb{C}^n$ be the $n$-dimensional complex Euclidian space. For any two points $z=(z_1,\cdots, z_n)$ and $w=(w_1,\cdots,w_n)$ in $\mathbb{C}^n$, we write $\langle z,w\rangle=\sum_{i=1}^nz_i\cdot \overline{w}_i$ and $|z|=\langle z,z\rangle^{1/2}$. Let $\mathbb{B}_n$ denote the open unit ball in $\mathbb{C}^n$. For $\alpha>-1$, put
\[dv_{\alpha}(z):=c_{\alpha}(1-|z|^2)^{\alpha}dv(z),\]
where $dv$ is the normalized volume measure on $\mathbb{B}_n$ and $c_{\alpha}=\frac{\Gamma(n+1+\alpha)}{n!\Gamma(\alpha+1)}$ is chosen such that $v_{\alpha}(\mathbb{B}_n)=1$. Denote by $H(\mathbb{B}_n)$ the set of all holomorphic functions on $\mathbb{B}_n$. For $0<p<\infty$, the standard weighted Bergman space $A_{\alpha}^p(\mathbb{B}_n)$ consists of all $f\in H(\mathbb{B}_n)$ such that
\[\|f\|_{p}:=\left(\int_{\mathbb{B}_n}|f(z)|^pdv_{\alpha}(z)\right)^{1/p}\]
is finite. It is known that $A_{\alpha}^p(\mathbb{B}_n)$ is a Banach space for $1\leq p<\infty$. In particular, $A_{\alpha}^2(\mathbb{B}_n)$ is a Hilbert space with inner product
\[\langle f,g\rangle=\int_{\mathbb{B}_n}f(z)\overline{g(z)}dv_{\alpha}(z),\quad f,g\in A_{\alpha}^2(\mathbb{B}_n).\]
Here and in the sequel, we use the same symbol $\langle \cdot,\cdot\rangle$ to denote the inner product in $\mathbb{C}^n$ and $A_{\alpha}^2(\mathbb{B}_n)$, and the readers should have no trouble recognizing the meaning from the context.

Denote by $S(\mathbb{B}_n)$ the set of all holomorphic self-maps of $\mathbb{B}_n$. Given $u\in H(\mathbb{B}_n)$ and $\varphi\in S(\mathbb{B}_n)$, the weighted composition operator $C_{u,\varphi}$ on $H(\mathbb{B}_n)$ is defined by
\[C_{u,\varphi}f=u\cdot f\circ\varphi.\]
It is known that weighted composition operators are closely related to the isometries on classical Hardy or Bergman spaces. See for example \cite{Ff,Kc}. When $u=1$, it reduces to the composition operator $C_{\varphi}$. The relationship between the operator-theoretic properties of $C_{\varphi}$ and the function-theoretic properties of $\varphi$ has been studied extensively during the past several decades. One can refer to \cite{CcMb, Sj} for various aspects on the theory of composition operators.

Motivated by the study of isolation phenomena in the space of composition operators, see \cite{SjSc}, study on differences of (weighted) composition operators has been of growing interest. The effort to characterize compact difference of two composition operators $C_{\varphi}-C_{\psi}$ on standard weighted Bergman spaces was initiated by Moorhouse \cite{Mj} and then by Saukko \cite{Se1,Se2}. Koo and Wang \cite{KhWm} introduced the notion of joint-Carleson measure and obtained a Carleson measure characterization for the difference $C_{\varphi}-C_{\psi}$ on $A_{\alpha}^p(\mathbb{B}_n)$. In 2007, Acharyya and Wu \cite{AsWz} first obtained a compact criteria for the difference $C_{u,\varphi}-C_{v,\psi}$ on weighted Bergman spaces, where the weights $u,v$ satisfy a certain growth condition. Recently, Choe et al. \cite{CbCkKhYj1,CbCkKhYj2} completely characterized the boundedness and compactness of differences $C_{u,\varphi}-C_{v,\psi}$ between weighted Bergman spaces over the unit disk in terms of Carleson measures. And then, they extended the results to the ball setting, see \cite{CbCkKhPi}. For more results about the difference of (weighted) composition operators on various settings, one can see \cite{Cj,CbKhSw,CbKhWm,LbRjWf,Pi,YzZz} and references therein. Following this line of research, in this paper, we first obtain an essential norm estimate for the difference $C_{u,\varphi}-C_{v,\psi}$ on $A_{\alpha}^p(\mathbb{B}_n)$. Our methods to estimate the lower bound are new and simpler.

To state our main results, we first introduce several natations. For $\varphi,\psi\in S(\mathbb{B}_n)$, put
\[\rho(z)=d(\varphi(z),\psi(z)),\quad z\in\mathbb{B}_n,\]
where $d$ denotes the pseudo-hyperbolic distance on $\mathbb{B}_n$, see Section 2.1. Given a positive Borel measure $\mu$ on $\mathbb{B}_n$ and $\varphi\in S(\mathbb{B}_n)$, the pull-back measure $\mu\circ\varphi^{-1}$ is defined by 
\[\mu\circ\varphi^{-1}(E)=\mu(\varphi^{-1}(E))\]
for any Borel set $E\subset \mathbb{B}_n$. For $u,v\in H(\mathbb{B}_n)$, $\varphi,\psi\in S(\mathbb{B}_n)$ and $\beta>0$, we now define several pull-back measures as follows:
\begin{equation*}
	\omega_{\varphi,u}^p:=\left(|\rho u|^pdv_{\alpha}\right)\circ\varphi^{-1},\quad \omega_{\psi,v}^p:=\left(|\rho v|^pdv_{\alpha}\right)\circ\psi^{-1},
\end{equation*}
and 
\begin{equation*}
	\sigma_{\varphi,\beta}^p:=\left(|u-v|^p(1-\rho)^{\beta}dv_{\alpha}\right)\circ\varphi^{-1},\quad \sigma_{\psi,\beta}^{p}:=\left(|u-v|^p(1-\rho)^{\beta}dv_{\alpha}\right)\circ\psi^{-1}.
\end{equation*}
For simplicity, put
\[\omega_p:=\omega_{\varphi,u}^p+\omega_{\psi,v}^p \quad {\rm and}\quad \sigma_{\beta,p}:=\sigma_{\varphi,\beta}^p+\sigma_{\psi,\beta}^p\]
Given a positive Borel measure $\mu$ on $\mathbb{B}_n$ and $s\in (0,1)$, we put 
\[M_s(\mu)(z)=\frac{\mu(E(z,s))}{(1-|z|^2)^{n+1+\alpha}}\]
for the averaging function of $\mu$, where $E(z,s)$ is the pseudo-hyperbolic ball with center $z\in\mathbb{B}_n$ and radius $s\in (0,1)$, see Section 2.1.

Now our first main result of this paper is stated as follows.

\begin{theorem}\label{theorem1.1}
	Let $1<p<\infty$, $\alpha>-1$ and $\beta>n+1+\alpha$. Suppose $u,v\in H(\mathbb{B}_n)$, $\varphi,\psi\in S(\mathbb{B}_n)$ and $C_{u,\varphi}-C_{v,\psi}$ is bounded on $A_{\alpha}^p(\mathbb{B}_n)$, then 
	\begin{equation*}
		\|C_{u,\varphi}-C_{v,\psi}\|_{e}^p\simeq \limsup_{|z|\to 1}\left[M_{s}(\omega_p)(z)+M_s(\sigma_{\beta,p})(z)\right]
	\end{equation*}
	for some (or any) $s\in (0,1)$.
\end{theorem}

For $0<p<\infty$, we recall that a compact operator $T$ acting on a separated Hilbert space $\mathcal{H}$ belongs to the Schatten-$p$ class $S_p(\mathcal{H})$ if its sequence of singular numbers belongs to the sequence space $l^p$, where the singular numbers are the square roots of the eigenvalues of the positive operator $T^*T$. $T$ is said to be Hilbert-Schmidt if $T\in S_2(\mathcal{H})$. One can refer to \cite[Chapter 1]{Zk} for some details about Schatten class. Choe et al. \cite{CbHtKh} characterized the Hilbert-Schmidt property of the difference $C_{\varphi}-C_{\psi}$ on weighted Bergman spaces over the unit disk and showed that $C_{\varphi}$ and $C_{\psi}$ are in the same path component in the space of composition operators under the Hilbert-Schmidt norm if and only if $C_{\varphi}-C_{\psi}$ is Hilbert-Schmidt. This result was then extended to the ball setting by Zhang and Zhou \cite{ZlZz}. In \cite{AsWz}, Acharyya and Wu first characterized Hilbert-Schmidt difference of two weighted composition operators $C_{u,\varphi}-C_{v,\psi}$ acting from Bergman space or Hardy space to an $L^2(\mu)$ space. To the best of our knowledge, there are few papers studying the Schatten-$p$ class difference of weighted composition operators. So our second aim in this paper is to give some characterizations for $C_{u,\varphi}-C_{v,\psi}$ belonging to $S_p(A_{\alpha}^2)$.

Let 
\[d\lambda(z)=\frac{dv(z)}{(1-|z|^2)^{n+1}}\]
be the M$\rm\ddot{o}$bius invariant measure on $\mathbb{B}_n$. Our second main result of this paper is stated as follows. 

\begin{theorem}\label{theorem1.2}
	Let $\alpha>-1$ and $\beta\geq 2(n+1+\alpha)$. Suppose $u,v\in H(\mathbb{B}_n)$, $\varphi,\psi\in S(\mathbb{B}_n)$ and $C_{u,\varphi}-C_{v,\psi}$ is bounded on $A_{\alpha}^2(\mathbb{B}_n)$. If $M_s(\omega_2)+M_s(\sigma_{\beta,2})$ belongs to $L^{\frac{p}{2}}(\mathbb{B}_n,d\lambda)$ for some (or any) $s\in (0,1)$, then $C_{u,\varphi}-C_{v,\psi}\in S_p(A_{\alpha}^2)$. Moreover, when $p\geq 2$, the above condition is also necessary for $C_{u,\varphi}-C_{v,\psi}\in S_p(A_{\alpha}^2)$.
\end{theorem}

The paper is organized as follows. In Section 2, we collect some preliminary facts and auxiliary lemmas that will be used later. In particular, we introduce the notion of $R$-Berezin transform to describe Schatten class operators. Section 3 is devoted to the proof of Theorem 1.1, which is divided into two parts, see Theorem 3.1 and Theorem 3.2. The proof of Theorem 1.2 is presented in Section 4. Moreover, we obtain a Schatten class charaterization for the difference $C_{u,\varphi}-C_{v,\psi}$ via simple Reproducing Kernel Thesis.

Throughout the paper we use the same letter $C$ to denote positive constants which may vary at different occurrences but do not depend on the essential argument. For non-negative quantities $A$ and $B$, we write $A\lesssim B$ (or equivalent $B\gtrsim A$) if there exists an absolute constant $C>0$ such that $A\leq CB$. $A\simeq B$ means both $A\lesssim B$ and $B\lesssim A$. 

\section{Preliminaries}

In this section we recall some basic facts and present some auxiliary lemmas which will be used in the sequel.

\subsection{Pseudo-hyperbolic Distance}

Given $z\in\mathbb{B}_n$, let $\Phi_z$ be the involutive automorphism of $\mathbb{B}_n$ that exchange $0$ and $z$. More explicitly, 
\begin{equation*}
	\Phi_{z}(w)=\frac{z-P_z(w)}{1-\langle w,z\rangle}+\sqrt{1-|z|^2}\frac{P_z(w)-w}{1-\langle w,z\rangle},\quad w\in\mathbb{B}_n,
\end{equation*}
where $P_z$ is the orthogonal projection from $\mathbb{C}^n$ onto the one dimensional subspace generated by $z$. The pseudo-hyperbolic distance between $z,w\in\mathbb{C}^n$ is given by
\[d(z,w)=\left|\Phi_z(w)\right|.\]
It is easy to see that
\begin{equation}\label{equa2.1}
	1-d^2(z,w)=\frac{(1-|z|^2)(1-|w|^2)}{|1-\langle z,w\rangle|^2}.
\end{equation}
The pseudo-hyperbolic ball centered at $z\in\mathbb{B}_n$ with radius $s\in (0,1)$ is defined by
\[E(z,s)=\left\{w\in\mathbb{B}_n:d(z,w)<s\right\}.\]
Given $s\in (0,1)$, it is well-known that
\begin{equation}\label{equa2.2}
	v_{\alpha}(E(z,s))\simeq (1-|z|^2)^{n+1+\alpha},
\end{equation}
and 
\begin{equation}\label{equa2.3}
	1-|z|^2\simeq 1-|w|^2\simeq |1-\langle z,w\rangle|
\end{equation}
for all $z\in\mathbb{B}_n$ and $w\in E(z,s)$. Moreover,
\begin{equation}\label{equa2.4}
	\left|1-\langle z,a\rangle\right|\simeq \left|1-\langle w,a\rangle\right|
\end{equation}
for all $a$, $z$ and $w$ in $\mathbb{B}_n$ with $d(z,w)<s$. Here, all the constants suppressed in \eqref{equa2.2}, \eqref{equa2.3} and \eqref{equa2.4} depend only on $s$ and $\alpha$.

The following two lemmas are essential in the study of the difference of weighted composition operators, which can be found in \cite{KhWm} and \cite{TcYzZz}.

\begin{lemma}[\cite{KhWm}]\label{lemma2.1}
	Let $0<p<\infty$, $\alpha>-1$ and $0<r<s<1$. Then there exists a constant $C=C(\alpha,p,s,r)>0$ such that
	\begin{equation*}
		|f(z)-f(w)|^p\leq C\frac{d(z,w)^p}{(1-|z|^2)^{n+1+\alpha}}\int_{E(z,s)}|f(\zeta)|^pdv_{\alpha}(\zeta)
	\end{equation*}
	for all $z\in\mathbb{B}_n$, $w\in E(z,r)$ and $f\in H(\mathbb{B}_n)$.
\end{lemma}

For $1\leq j\leq n$, let $e_j=(0,\dots,0,1,0,\dots,0)$, where $1$ is on the $j$-th component. For any $a\in\mathbb{B}_n\backslash \{0\}$, choose an unitary transform $U_a$ such that $U_aa=|a|e_1$. Denote by $$a^j=U_a^*(|a|e_j),\quad j=2,\dots,n.$$ 

\begin{lemma}[\cite{TcYzZz}]\label{lemma2.2}
	Let $s\in (0,1)$. There exists $0<t_0<1$ such that
	\begin{equation*}
		d(z,w)\simeq \frac{1}{|1-\langle z,w\rangle|}\left(\left|\langle z-w,a\rangle\right|+\sqrt{1-|a|}\sum_{j=2}^n|\langle z-w,a^j\rangle|\right)
	\end{equation*}
	for all $a\in\mathbb{B}_n$ with $t_0<|a|<1$, $z\in E(a,s)$ and $w\in\mathbb{B}_n$.
\end{lemma}

\subsection{Test Functions}

Given $\alpha>-1$ and $s\in (0,1)$, we recall the following pointwise estimate
\begin{equation}\label{equa2.5}
	|f(z)|^p\lesssim  \frac{1}{(1-|z|^2)^{n+1+\alpha}}\int_{E(z,s)}|f(\zeta)|^pdv_{\alpha}(\zeta),\quad z\in\mathbb{B}_n
\end{equation}
valid for functions $f$ in $A_{\alpha}^p(\mathbb{B}_n)$ and $0<p<\infty$, where constant suppressed depends only on $n$, $\alpha$ and $s$, see for example \cite[Lemma 2.24]{Zk1}.

Note from \eqref{equa2.5} that each point evaluation: $f\mapsto f(z)$ is a continuous linear functional on $A_{\alpha}^p(\mathbb{B}_n)$ for every $z\in\mathbb{D}$. When $p=2$, by the Riesz representation theorem in Hilbert space theory, to each $z\in\mathbb{B}_n$ corresponds a unique function $K_z\in A_{\alpha}^2(\mathbb{B}_n)$ such that
\begin{equation*}
	f(z)=\langle f,K_z\rangle,\quad \forall f\in A_{\alpha}^2(\mathbb{B}_n).
\end{equation*}
$K_z$ is called the reproducing kernel function in $A_{\alpha}^2(\mathbb{B}_n)$ at $z$, whose explicit formula is known as
\[K_z(w)=\frac{1}{(1-\langle w,z\rangle)^{n+1+\alpha}},\quad w\in\mathbb{B}_n.\]
A very important concept of the differentiation in several complex variables is that of the radial derivative, which is based on the usual partial derivatives of a holomorphic function. Specifically, we write 
\[Rf(z)=\sum_{j=1}^nz_j\frac{\partial f}{\partial z_j}(z),\quad z\in\mathbb{B}_n,\]
for the radial derivative of $f$.

For any $z\in\mathbb{B}_n$, let
\[K_z^{[R]}(w)=\sum_{j=1}^n\overline{z}_j\frac{\partial K_z}{\partial \overline{z}_j}(w),\quad w\in\mathbb{B}_n.\]
Then for any $f\in A_{\alpha}^2(\mathbb{B}_n)$, we have
\begin{equation}\label{equa2.6}
	Rf(z)=\langle f,K_z^{[R]}\rangle_2.
\end{equation}
It is known that $\|K_z\|_2^2= \frac{1}{(1-|z|^2)^{n+1+\alpha}}$ and 
\[\|K_z^{[R]}\|_{2}^2= (RK_z^{[R]})(z)\simeq \frac{1}{(1-|z|^2)^{n+3+\alpha}}.\]

Let $0<p<\infty$, $\alpha>-1$ and $N$ be a positive integer such that $N>\frac{n+1+\alpha}{p}$. For any $a\in\mathbb{B}_n$, let 
\[F_{a,N}(z)=\frac{1}{(1-\langle z,a\rangle)^{N}},\quad z\in\mathbb{B}_n.\]
By \cite[Theorem 1.12]{Zk1}, we know that $\|F_{a,N}\|_{p}\simeq (1-|a|^2)^{\frac{n+1+\alpha}{p}-N}$ and constant suppressed is independent of $a$. We set 
\[f_{a,N,p}(z)=(1-|a|^2)^{N-\frac{n+1+\alpha}{p}}F_{a,N}(z).\]
Then $\|f_{a,N,p}\|_p\simeq 1$ and $f_{a,N,p}$ converges to 0 uniformly on compact subsets of $\mathbb{B}_n$ as $|a|\to 1$.

\subsection{Compact operator and Schatten-$p$ class}

Given two complete metrizable topological vector spaces $X$ and $Y$, a bounded linear operator $T:X\to Y$ is said to be compact if the image of each bounded sequence in $X$ has a convergent subsequence in $Y$. 

The following lemma derives from \cite[Lemma 4.1]{Se1}, which is needed to obtain the lower bounds of essential norm estimates for the difference of two weighted composition operators.

\begin{lemma}\label{lemma2.3}
	Let $X$, $Y$ be two Banach spaces and $T:X\to Y$ be a bounded linear operator. Suppose $\{f_m\}_{m=1}^{\infty}$ is a sequence in $X$ such that $\|f_m\|_{X}=1$ for each $m\in\mathbb{N}$ and $f_m\to 0$ weakly as $m\to \infty$, then
	\[\|T\|_e\geq \limsup_{m\to\infty}\|Tf_m\|_Y.\]
\end{lemma}

For any $z\in\mathbb{D}$, let $k_z=K_z/\|K_z\|_{2}$. Suppose $T$ is a bounded linear operator on $A_{\alpha}^2(\mathbb{B}_n)$, recall that the Berezin  transform of $T$ is defined by 
\[\widetilde{T}(z)=\langle Tk_z, k_z\rangle,\quad z\in\mathbb{B}_n.\]
Similarly, we can define a class of $R$-Berezin transform $\widetilde{T}_{[R]}$ as follows:
\[\widetilde{T}_{[R]}(z)=\langle Tk_z^{[R]},k_z^{[R]}\rangle,\quad z\in\mathbb{B}_n,\]
where $k_z^{[R]}=K_z^{[R]}/\|K_z^{[R]}\|_2$.

\begin{lemma}\label{lemma2.4}
	Suppose $T$ is a positive operator on $A_{\alpha}^2(\mathbb{B}_n)$. If $T$ belongs to the trace class, then $\widetilde{T}_{[R]}\in L^1(\mathbb{B}_n,d\lambda)$ and $\|T\|_{S_1}\gtrsim \|\widetilde{T}_{[R]}\|_{L^1(\mathbb{B}_n,d\lambda)}$.
\end{lemma}

\begin{proof}
	Fix an orthonormal basis $\{e_j\}$ for $A_{\alpha}^2(\mathbb{B}_n)$. Since $T$ is positive, it belongs to the trace class if and only if 
	\[\sum_{j=1}^{\infty}\langle Te_j,e_j\rangle<\infty.\]
	
	On the other hand, let $S=\sqrt{T}$. By Fubini's Theorem, \cite[Theorem 2.16]{Zk1} and the reproducing formula \eqref{equa2.6}, we obtain
	\begin{equation*}
		\begin{split}
			\|T\|_{S_1}&=\sum_{j=1}^{\infty}\langle Te_j,e_j\rangle=\sum_{j=1}^{\infty}\|Se_j\|_{2}^2\\
			&\gtrsim \sum_{j=1}^{\infty}\int_{\mathbb{B}_n}\left|R(Se_j)(z)\right|^2(1-|z|^2)^{2}dv_{\alpha}(z)\\
			&=\int_{\mathbb{B}_n}\left(\sum_{j=1}^{\infty}\langle Se_j,K_z^{[R]}\rangle^2\right)(1-|z|^2)^{2}dv_{\alpha}(z)\\
			&=\int_{\mathbb{B}_n}\left(\sum_{j=1}^{\infty}\langle e_j, SK_z^{[R]}\rangle^2\right)(1-|z|^2)^{2}dv_{\alpha}(z)\\
			&=\int_{\mathbb{B}_n}\|SK_z^{[R]}\|_{2}^2(1-|z|^2)^{2}dv_{\alpha}(z)\\
			&=\int_{\mathbb{B}_n}\langle TK_z^{[R]}, K_z^{[R]}\rangle(1-|z|^2)^{2}dv_{\alpha}(z)\\
			&\simeq \int_{\mathbb{B}_n}\widetilde{T}_{[R]}(z) d\lambda(z).
		\end{split}
	\end{equation*}
	Therefore, trace class of $T$ implies that $\widetilde{T}_{[R]}\in L^1(\mathbb{B}_n,d\lambda)$. The proof is complete.
\end{proof}

\begin{lemma}\label{lemma2.5}
	Suppose $T$ is a positive operator on $A_{\alpha}^2(\mathbb{B}_n)$ and $1\leq p<\infty$. If $T\in S_p(A_{\alpha}^2)$, then $\widetilde{T}\in L^p(\mathbb{B}_n,d\lambda)$ and $\widetilde{T}_{[R]}\in L^{p}(\mathbb{B}_n,d\lambda)$. Moreover, $\|T\|_{S_p}\gtrsim \|\widetilde{T}\|_{L^p(\mathbb{B}_n,d\lambda)}+ \|\widetilde{T}_{[R]}\|_{L^p(\mathbb{B}_n,d\lambda)}$.
\end{lemma}

\begin{proof}
	Assume $T$ belongs to $S_p(A_{\alpha}^2)$. Firstly, it follows from \cite[Lemma 2]{Pj} that $\widetilde{T}\in L^p(\mathbb{B}_n,d\lambda)$. Then by \cite[Proposition 1.31]{Zk}, we have
	\[(\widetilde{T^p})_{[R]}(z)\geq \left(\widetilde{T}_{[R]}(z)\right)^p,\quad z\in\mathbb{B}_n.\]
	Since $T$ is positive, $T\in S_p$ implies that $T^p$ is in the trace class. So it follows from Lemma \ref{lemma2.4} that 
	\[\int_{\mathbb{B}_n}\left(\widetilde{T}_{[R]}(z)\right)^pd\lambda(z)\leq \int_{\mathbb{B}_n}(\widetilde{T^p})_{[R]}(z)d\lambda(z)\lesssim \|T^p\|_{S_1}=\|T\|_{S_p}<\infty.\]
	This completes the proof.
\end{proof}

\subsection{Carleson Measure and Toeplitz Operator}

Let $\mu$ be a positive Borel measure on $\mathbb{B}_n$. For $\alpha>-1$ and $0<p<\infty$, we say that $\mu$ is a Carleson measure for $A_{\alpha}^p(\mathbb{B}_n)$ if there exists a constant $C>0$ such that
\[\left(\int_{\mathbb{B}_n}|f(z)|^pd\mu(z)\right)^{1/p}\leq C\|f\|_{p}\]
for all $f\in A_{\alpha}^p(\mathbb{B}_n)$. That is, $\mu$ is a Carleson measure for $A_{\alpha}^p(\mathbb{B}_n)$ if the embedding $A_{\alpha}^p(\mathbb{B}_n)\subset L^p(\mathbb{B}_n,d\mu)$ is continuous. If, in addition, the embedding $A_{\alpha}^p(\mathbb{B}_n)\subset L^p(\mathbb{B}_n,d\mu)$ is compact, then $\mu$ is called a compact Carleson measure for $A_{\alpha}^p(\mathbb{B}_n)$. Note that Carleson measures are finite.

Given a positive Borel measure $\mu$ on $\mathbb{B}_n$ and $s\in (0,1)$, recall that the averaging function of $\mu$ is defined by 
\[M_{s}(\mu)(z)=\frac{\mu(E(z,s))}{(1-|z|^2)^{n+1+\alpha}},\quad z\in\mathbb{B}_n.\]
It is known by \cite[Theorem 50]{ZrZk} that $\mu$ is a Carleson measure for $A_{\alpha}^p(\mathbb{B}_n)$ if and only if
\[\sup_{z\in\mathbb{B}_n}M_s(\mu)(z)<\infty\]
for some (or any) $s\in (0,1)$. And $\mu$ is a compact Carleson measure for $A_{\alpha}^p(\mathbb{B}_n)$ if and only if
\[\lim_{|z|\to 1^{-}}M_s(\mu)(z)=0\]
for some (or any) $s\in (0,1)$. Moreover, characterizations of Carleson measures are independent of $p$.

For a function $f\in A_{\alpha}^p(\mathbb{B}_n)$ and $s\in (0,1)$, by \eqref{equa2.5} and Fubini's Theorem, it is easy to see that
\begin{equation}\label{equa2.7}
	\int_{\mathbb{B}_n}|f(z)|^pd\mu(z)\lesssim \int_{\mathbb{B}_n}|f(w)|^pM_{s}(\mu)(w)dv_{\alpha}(w).
\end{equation}

Given a positive Borel measure $\mu$ on $\mathbb{B}_n$, the Toeplitz operator $T_{\mu}$ acting on $A_{\alpha}^2(\mathbb{B}_n)$ is the densely defined integral operator
\[T_{\mu}f(z)=\int_{\mathbb{D}}\frac{f(w)}{(1-\langle w,z\rangle)^{n+1+\alpha}}d\mu(w), \quad z\in\mathbb{D}.\]
Clearly, $T_{\mu}$ is well defined on $H^{\infty}(\mathbb{B}_n)$, the space of all bounded analytic functions in $\mathbb{D}$. 
According to \cite[Theorem 17]{Zk2}, we know that $T_{\mu}$ is bounded (compact, resp.) on $A_{\alpha}^2(\mathbb{B}_n)$ if and only if $\mu$ is a (compact, resp.) Carleson measure. And for $0<p<\infty$, $T_{\mu}$ is in $S_p(A_{\alpha}^2)$ if and only if 
\[M_s(\mu)\in L^{p}(\mathbb{B}_n,d\lambda)\] for some (or any) $0<s<1$ and $\|T_{\mu}\|_{S_p}\simeq \|M_s(\mu)\|_{L^p(\mathbb{D},d\lambda)}$.

\begin{lemma}\label{lemma2.6}
	Suppose $\mu$ is a compact Carleson measure and $0<s<r<1$. Let $d\mu_1=M_s(\mu)dv_{\alpha}$. If $M_{r}(\mu)\in L^p(\mathbb{B}_n,d\lambda)$, then $T_{\mu_1}$ is in $S_p(A_{\alpha}^2)$ and $\|T_{\mu_1}\|_{S_p}\lesssim \|M_{r}(\mu)\|_{L^p(\mathbb{B}_n,d\lambda)}$.
\end{lemma}
\begin{proof}
	By \eqref{equa2.3}, it is easy to see that $M_{r-s}(\mu_1)\lesssim M_{r}(\mu)$. So $M_{r-s}(\mu_1)\in L^p(\mathbb{B}_n,d\lambda)$, and then $T_{\mu_1}\in S_p(A_{\alpha}^2)$. Moreover, 
	$$\|T_{\mu_1}\|_{S_p}\simeq \|M_{r-s}(\mu_1)\|_{L^p(\mathbb{B}_n,d\lambda)}\lesssim \|M_r(\mu)\|_{L^p(\mathbb{B}_n,d\lambda)}.$$
	The proof is complete.
\end{proof}

\section{Proof of Theorem 1.1}

In this section, we provide the proof of Theorem 1.1, which gives the essential norm estimates for the difference of two weighted composition operators on $A_{\alpha}^p(\mathbb{B}_n)$.

\subsection{Upper Bounds}

Let $\{t_m\}_{m=1}^{\infty}$ be a fixed real sequence in $(0,1)$ such that $t_1<t_2<\cdots$ and $t_m\to 1$ as $m\to\infty$. For $m\in\mathbb{N}$, define
\[(T_mf)(z):=f(t_mz),\quad f\in H(\mathbb{B}_n),~~z\in\mathbb{B}_n.\]
It is easy to see that each $T_m$ is bounded on $A_{\alpha}^p(\mathbb{B}_n)$ and $\|T_m\|=1$. Moreover, $T_m$ is compact for each $m$. If we set $R_m=I-T_m$, then $R_m$ is also bounded on $A_{\alpha}^p(\mathbb{B}_n)$ and $\|R_m\|\leq 2$.

Now we provide the upper estimate for the essential norm of the operator $C_{u,\varphi}-C_{v,\psi}$.

\begin{theorem}\label{theorem3.1}
	Let $0<p<\infty$ and $\alpha>-1$. Suppose $u,v\in H(\mathbb{B}_n)$, $\varphi,\psi\in S(\mathbb{B}_n)$ and $C_{u,\varphi}-C_{v,\psi}$ is bounded on $A_{\alpha}^p(\mathbb{B}_n)$. Then 
	\begin{equation*}
		\|C_{u,\varphi}-C_{v,\psi}\|_e^p\lesssim \lim_{\delta\to 1^-}\sup_{|z|>\delta}\left[M_{s}(\omega_p)(z)+M_s(\sigma_{\beta,p})(z)\right]
	\end{equation*}
	for some (or any) $s\in (0,1)$.
\end{theorem}

\begin{proof}
	Fix $0<r<s<1$. Let $G_r=\{z\in\mathbb{B}_n:\rho(z)<r\}$. Assume $C_{u,\varphi}-C_{v,\psi}$ is bounded on $A_{\alpha}^p(\mathbb{B}_n)$. Since $T_m$ is compact on $A_{\alpha}^p(\mathbb{B}_n)$, so is $(C_{u,\varphi}-C_{v,\psi})T_m$. Hence,
	\begin{equation*}
		\begin{split}
			\|C_{u,\varphi}-C_{v,\psi}\|_e^p&\leq \overline{\lim_{m\to\infty}}\|(C_{u,\varphi}-C_{v,\psi})\circ R_m\|^p\\
			&=\overline{\lim_{m\to\infty}}\sup_{\|f\|_{p}\leq 1}\|(C_{u,\varphi}-C_{v,\psi})\circ R_mf\|_p^p\\
			&\lesssim \overline{\lim_{m\to\infty}}\sup_{\|f\|_p\leq 1}\left[I_{1,m}(f)+I_{2,m}(f)+I_{3,m}(f)\right],
		\end{split}
	\end{equation*}
	where
	\begin{equation*}
		I_{1,m}(f)=\int_{\mathbb{B}_n\backslash G_r}|u(z)R_mf(\varphi(z))-v(z)R_mf(\psi(z))|^pdv_{\alpha}(z),
	\end{equation*}
	\begin{equation*}
		I_{2,m}(f)=\int_{G_r}|u(z)-v(z)|^p\left(|R_mf(\varphi(z))|^p+|R_mf(\psi(z))|^p\right)dv_{\alpha}(z)
	\end{equation*}
	and 
	\begin{equation*}
		I_{3,m}(f)=\int_{G_r}\left(|u(z)|^p+|v(z)|^p\right)\left|R_mf(\varphi(z))-R_mf(\psi(z))\right|^pdv_{\alpha}(z).
	\end{equation*}
	Now we estimate each integral separately. For the integral $I_{1,m}(f)$, note that $\rho(z)\geq r$ for $z\in\mathbb{B}_n\backslash G_r$. By \eqref{equa2.7}, we obtain
	\begin{equation}\label{equa3.1}
		\begin{split}
			I_{1,m}(f)&\lesssim \int_{\mathbb{B}_n\backslash G_r}\rho(z)^p\left|u(z)R_mf(\varphi(z))-v(z)R_mf(\psi(z))\right|^pdv_{\alpha}(z)\\
			&\lesssim \int_{\mathbb{B}_n}\left[|u\rho|^p|(R_mf)\circ\varphi|^p+|v\rho|^p|(R_mf)\circ\psi|^p\right]dv_{\alpha}\\
			&=\int_{\mathbb{B}_n}|R_mf|^pd\omega_p\\
			&\lesssim \int_{\mathbb{B}_m}|R_mf|^pM_s(\omega_p)dv_{\alpha}.
		\end{split}
	\end{equation}
	For the integral $I_{2,m}(f)$, note that $1-\rho(z)\geq 1-r$ for $z\in G_r$. By \eqref{equa2.7}, we obtain
	\begin{equation}\label{equa3.2}
		\begin{split}
			I_{2,m}(f)&\lesssim \int_{G_r}(1-\rho(z))^{\beta}|u(z)-v(z)|^p\left(|R_mf(\varphi(z))|^p+|R_mf(\psi(z))|^p\right)dv_{\alpha}(z)\\
			&\lesssim \int_{\mathbb{B}_n}|R_mf(z)|^pd\sigma_{\beta,p}(z)\\
			&\lesssim \int_{\mathbb{B}_n}|R_mf(z)|^pM_s(\sigma_{\beta,p})(z)dv_{\alpha}(z).
		\end{split}
	\end{equation}
	It remains to estimate the integral $I_{3,m}(f)$. Since each $R_m$ is bounded, for $z\in G_r$, by Lemma \ref{lemma2.1}, we have
	\begin{equation*}
		\begin{split}
			&\left|R_mf(\varphi(z))-R_mf(\psi(z))\right|^p\lesssim \frac{\rho(z)^p}{(1-|\varphi(z)|^2)^{n+1+\alpha}}\int_{E(\varphi(z),s)}|R_mf(\zeta)|^pdv_{\alpha}(\zeta).
		\end{split}
	\end{equation*}
	Then it follows from \eqref{equa2.3} that
	\begin{equation*}
		\begin{split}
			&\int_{G_r}|u(z)|^p|R_mf(\varphi(z))-R_mf(\psi(z))|^pdv_{\alpha}(z)\\
			&\quad \lesssim \int_{G_r}\frac{|u(z)|^p\rho(z)^p}{(1-|\varphi(z)|^2)^{n+1+\alpha}}\left(\int_{E(\varphi(z),s)}|R_mf(\zeta)|^pdv_{\alpha}(\zeta)\right)dv_{\alpha}(z)\\
			&\quad \lesssim \int_{\mathbb{B}_n}|R_mf(\zeta)|^p\frac{\int_{\varphi^{-1}(E(\zeta,s))}|u(z)|^p\rho(z)^pdv_{\alpha}(z)}{(1-|\zeta|^2)^{n+1+\alpha}}dv_{\alpha}(\zeta)\\
			&\quad =\int_{\mathbb{B}_n}|R_mf(\zeta)|^pM_s(\omega_{\varphi,u}^p)(\zeta)dv_{\alpha}(\zeta).
		\end{split}
	\end{equation*}
	Similarly,
	\begin{equation*}
		\begin{split}
			&\int_{G_r}|v(z)|^p|R_mf(\varphi(z))-R_mf(\psi(z))|^pdv_{\alpha}(z)\lesssim \int_{\mathbb{B}_n}|R_mf(\zeta)|^pM_s(\omega_{\psi,v}^p)(\zeta)dv_{\alpha}(\zeta).
		\end{split}
	\end{equation*}
	Therefore,
	\begin{equation}\label{equa3.3}
		I_{3,m}(f)\lesssim \int_{\mathbb{B}_n}|R_mf(z)|^pM_s(\omega_p)(z)dv_{\alpha}(z).
	\end{equation}
	
	Combining \eqref{equa3.1}, \eqref{equa3.2} and \eqref{equa3.3}, we obtain
	\begin{equation*}
		\begin{split}
			&\|C_{u,\varphi}-C_{v,\psi}\|_e^p\lesssim \overline{\lim_{m\to\infty}}\sup_{\|f\|_p\leq 1}\int_{\mathbb{B}_n}|R_mf(z)|^p\left[M_s(\omega_p)(z)+M_s(\sigma_{\beta,p})(z)\right]dv_{\alpha}(z).
		\end{split}
	\end{equation*}
	By \cite[Theorem 4.5]{TcYzZz} or \cite[Theorem 1.1]{CbCkKhPi}, the boundedness of $C_{u,\varphi}-C_{v,\psi}$ implies that
	\[\sup_{z\in\mathbb{B}_n}M_s(\omega_p)(z)+M_s(\sigma_{\beta,p})(z)<\infty.\]
	Hence,
	\begin{equation*}
		\begin{split}
			&\overline{\lim_{m\to\infty}}\sup_{\|f\|_p\leq 1}\int_{B_{\delta}}|R_mf(z)|^p\left[M_s(\omega_p)(z)+M_s(\sigma_{\beta,p})(z)\right]dv_{\alpha}(z)\\
			&\quad \lesssim \overline{\lim_{m\to\infty}}\sup_{\|f\|_p\leq 1}\int_{B_{\delta}}|R_mf(z)|^pdv_{\alpha}(z)=0,
		\end{split}
	\end{equation*}
	since $R_mf$ converges to 0 uniformly on compact subsets of $\mathbb{B}_n$ as $m\to\infty$, where $0<\delta<1$ and $B_{\delta}=\left\{z:|z|\leq \delta\right\}$. Consequently, by the boundedness of $R_m$ ($\|R_m\|\leq 2$), we obtain
	\begin{equation*}
		\begin{split}
			\|C_{u,\varphi}-C_{v,\psi}\|_e^p&\lesssim \overline{\lim_{m\to\infty}}\sup_{\|f\|_p\leq 1}\int_{\mathbb{B}_n\backslash B_{\delta}}|R_mf(z)|^p\left[M_s(\omega_p)(z)+M_s(\sigma_{\beta,p})(z)\right]dv_{\alpha}(z)\\
			&\leq 2\sup_{|z|>\delta}M_s(\omega_p)(z)+M_s(\sigma_{\beta,p})(z).
		\end{split}
	\end{equation*}
	Letting $\delta\to 1^-$ yields
	\[\|C_{u,\varphi}-C_{v,\psi}\|_e^p\lesssim \lim_{\delta\to 1^-}\sup_{|z|>\delta}M_s(\omega_p)(z)+M_s(\sigma_{\beta,p})(z).\]
	The proof is complete.
\end{proof}

\subsection{Lower Bounds}

We proceed to provide the lower bound estimate for the essential norm of the operator $C_{u,\varphi}-C_{v,\psi}$.

\begin{theorem}\label{theorem3.2}
	Let $1<p<\infty$ and $\alpha>-1$. Suppose $u,v\in H(\mathbb{B}_n)$, $\varphi,\psi\in S(\mathbb{B}_n)$ and $C_{u,\varphi}-C_{v,\psi}$ is bounded on $A_{\alpha}^p(\mathbb{B}_n)$. Then 
	\begin{equation*}
		\|C_{u,\varphi}-C_{v,\psi}\|_e^p\gtrsim \overline{\lim_{|a|\to 1^-}}M_s(\omega_p)(a)+M_s(\sigma_{\beta,p})(a)
	\end{equation*}
	for some (or any) $s\in (0,1)$.
\end{theorem}

\begin{proof}
	For any $a\in\mathbb{B}_n$ and $N>\frac{n+1+\alpha}{p}$, recall the functions
	\[f_{a, N, p}(z)=\frac{(1-|a|^2)^{N-\frac{n+1+\alpha}{p}}}{(1-\langle z,a\rangle)^N},\quad z\in\mathbb{B}_n.\]
	It is easy to see that $f_{a,N,p}$ converges weakly to 0 in $A_{\alpha}^p(\mathbb{B}_n)$ as $|a|\to 1^-$. So by Lemma \ref{lemma2.3}, we get 
	\begin{equation*}
		\begin{split}
			&\|C_{u,\varphi}-C_{v,\psi}\|_e\\
			&\quad \gtrsim \overline{\lim_{|a|\to 1^-}}\|(C_{u,\varphi}-C_{v,\psi})f_{a,N,p}\|_p+\overline{\lim_{|a|\to 1^-}}\|(C_{u,\varphi}-C_{v,\psi})f_{a,N+1,p}\|_p.
		\end{split}
	\end{equation*}
	Fix $s\in (\frac{1}{4},\frac{1}{2})$. By \eqref{equa2.3}, we have
	\begin{equation}\label{equa3.4}
		\begin{split}
			&\|(C_{u,\varphi}-C_{v,\psi})f_{a,N,p}\|_p^p \gtrsim\frac{\int_{\varphi^{-1}(E(a,s))}\left|u(z)-v(z)\left(\frac{1-\langle \varphi(z),a\rangle}{1-\langle \psi(z),a\rangle}\right)^N\right|^pdv_{\alpha}(z)}{(1-|a|^2)^{n+1+\alpha}}.
		\end{split}
	\end{equation}
	Similarly,
	\begin{equation}\label{equa3.5}
		\begin{split}
			&\|(C_{u,\varphi}-C_{v,\psi})f_{a,N+1,p}\|_p^p \gtrsim\frac{\int_{\varphi^{-1}(E(a,s))}\left|u(z)-v(z)\left(\frac{1-\langle \varphi(z),a\rangle}{1-\langle \psi(z),a\rangle}\right)^{N+1}\right|^pdv_{\alpha}(z)}{(1-|a|^2)^{n+1+\alpha}}.
		\end{split}
	\end{equation}
	Note that $\left|\frac{1-\langle \varphi(z),a\rangle}{1-\langle \psi(z),a\rangle}\right|\lesssim 1$ when $\varphi(z)\in E(a,s)$ by \eqref{equa2.3}. Multiplying the integrand in \eqref{equa3.4} by $\left|\frac{1-\langle \varphi(z),a\rangle}{1-\langle \psi(z),a\rangle}\right|^p$, and adding it to \eqref{equa3.5}, by the triangle inequality, we obtain
	\begin{equation}\label{equa3.6}
		\begin{split}
			&\|(C_{u,\varphi}-C_{v,\psi})f_{a,N,p}\|_p^p+\|(C_{u,\varphi}-C_{v,\psi})f_{a,N+1,p}\|_{p}^p\\
			&\quad \gtrsim \int_{\varphi^{-1}(E(a,s))}|u(z)|^p\left|\frac{\langle \psi(z)-\varphi(z),a\rangle}{1-\langle \psi(z),a\rangle}\right|^p\frac{1}{(1-|a|^2)^{n+1+\alpha}}dv_{\alpha}(z).
		\end{split}
	\end{equation}
	By \eqref{equa2.4}, $|1-\langle \psi(z),a\rangle|\simeq |1-\langle \psi(z),\varphi(z)\rangle|$ when $\varphi(z)\in E(a,s)$. So we get
	\begin{equation}\label{equa3.7}
		\|C_{u,\varphi}-C_{v,\psi}\|_e^p\gtrsim \overline{\lim_{|a|\to 1^{-}}}\frac{\int_{\varphi^{-1}(E(a,s))}|u(z)|^p\left|\frac{\langle \psi(z)-\varphi(z),a\rangle}{1-\langle \psi(z),\varphi(z)\rangle}\right|^pdv_{\alpha}(z)}{(1-|a|^2)^{n+1+\alpha}}.
	\end{equation}
	
	Let 
	\[b^j=a+s\sqrt{1-|a|^2}a^j,\quad j=2,\cdots,n.\]
	Then $d(a,b^j)<s$ and $|b^j|\to 1^-$ whenever $|a|\to 1^-$. Moreover, $E(a,s)\subset E(b^j,2s)$. Applying a similar argument as above and by \eqref{equa2.3}, we obtain
	\begin{equation}\label{equa3.8}
		\begin{split}
			\|C_{u,\varphi}-C_{v,\psi}\|_e^p
			&\gtrsim \overline{\lim_{|a|\to 1^-}}\|(C_{u,\varphi}-C_{v,\psi})f_{b^j,N,p}\|_p^p+\overline{\lim_{|a|\to 1^-}}\|(C_{u,\varphi}-C_{v,\psi})f_{b^j,N+1,p}\|_p^p\\
			&\gtrsim \overline{\lim_{|a|\to 1^-}}\frac{\int_{\varphi^{-1}(E(b^j,2s))}|u(z)|^p\left|\frac{\langle \psi(z)-\varphi(z),b^j\rangle}{1-\langle \psi(z),\varphi(z)\rangle}\right|^pdv_{\alpha}(z)}{(1-|b^j|^2)^{n+1+\alpha}}\\
			&\gtrsim \overline{\lim_{|a|\to 1^-}}\frac{\int_{\varphi^{-1}(E(a,s))}|u(z)|^p\left|\frac{\langle \psi(z)-\varphi(z),b^j\rangle}{1-\langle \psi(z),\varphi(z)\rangle}\right|^pdv_{\alpha}(z)}{(1-|a|^2)^{n+1+\alpha}}.
		\end{split}
	\end{equation}
	Combining \eqref{equa3.7} and \eqref{equa3.8}, we get
	\begin{equation}\label{equa3.9}
		\begin{split}
			&\|C_{u,\varphi}-C_{v,\psi}\|_e^p
			\gtrsim  \overline{\lim_{|a|\to 1^-}}\frac{\int_{\varphi^{-1}(E(a,s))}|u(z)|^p\left|\frac{\langle \psi(z)-\varphi(z),\sqrt{1-|a|^2}a^j\rangle}{1-\langle \psi(z),\varphi(z)\rangle}\right|^pdv_{\alpha}(z)}{(1-|a|^2)^{n+1+\alpha}}.
		\end{split}
	\end{equation}
	for $j=2,\cdots,n$. Then combining \eqref{equa3.7} and \eqref{equa3.9}, and by Lemma \ref{lemma2.2}, we obtain
	\begin{equation}\label{equa3.10}
		\begin{split}
			\|C_{u,\varphi}-C_{v,\psi}\|_e^p&\gtrsim \overline{\lim_{|a|\to 1^-}}\frac{\int_{\varphi^{-1}(E(a,s))}|u(z)|^p\rho(z)^pdv_{\alpha}(z)}{(1-|a|^2)^{n+1+\alpha}}\\
			&=\overline{\lim_{|a|\to 1^-}}M_s(\omega_{\varphi,u}^p)(a).
		\end{split}
	\end{equation}
	Similarly, 
	\begin{equation}\label{equa3.11}
		\|C_{u,\varphi}-C_{v,\psi}\|_e^p\gtrsim\overline{\lim_{|a|\to 1^-}}M_s(\omega_{\psi,v}^p)(a).
	\end{equation}
	
	On the other hand, when $\varphi(z)\in E(a,s)$, by \eqref{equa2.1}, \eqref{equa2.3} and \eqref{equa2.4}, we have
	\begin{equation*}
		\left|\frac{1-\langle \varphi(z),a\rangle}{1-\langle \psi(z),a\rangle}\right|\gtrsim \frac{(1-|\varphi(z)|^2)(1-|\psi(z)|^2)}{|1-\langle \varphi(z),\psi(z)\rangle|^2}=1-\rho(z)^2
	\end{equation*}
	and
	\begin{equation*}
		\begin{split}
			&|u(z)-v(z)|\left|\frac{1-\langle \varphi(z),a\rangle}{1-\langle \psi(z),a\rangle}\right|^N\\
			&\quad \leq \left|u(z)-v(z)\left(\frac{1-\langle \varphi(z),a\rangle}{1-\langle \psi(z),a\rangle}\right)^N\right|+|u(z)|\left|1-\frac{1-\langle \varphi(z),a\rangle}{1-\langle \psi(z),a\rangle}\right|\\
			&\quad \lesssim\left|u(z)-v(z)\left(\frac{1-\langle \varphi(z),a\rangle}{1-\langle \psi(z),a\rangle}\right)^N\right|+|u(z)|\left|\frac{\langle \varphi(z)-\psi(z),a\rangle}{1-\langle\psi(z),a\rangle }\right|.
		\end{split}
	\end{equation*}
	So we combine \eqref{equa3.4} and \eqref{equa3.6} to obtain
	\begin{equation}\label{equa3.12}
		\begin{split}
			\|C_{u,\varphi}-C_{v,\psi}\|_e^p&\gtrsim\overline{\lim_{|a|\to 1^-}}\frac{\int_{\varphi^{-1}(E(a,s))}|u(z)-v(z)|^p(1-\rho(z))^{Np}dv_{\alpha}(z)}{(1-|a|^2)^{n+1+\alpha}}\\
			&\geq \overline{\lim_{|a|\to 1^-}}M_s(\sigma_{\varphi,\beta}^p)(a),
		\end{split}
	\end{equation}
	where $\beta\geq Np>n+1+\alpha$. Similarly,
	\begin{equation}\label{equa3.13}
		\|C_{u,\varphi}-C_{v,\psi}\|_e^p\gtrsim\overline{\lim_{|a|\to 1^-}}M_s(\sigma_{\psi,\beta}^p)(a).
	\end{equation}
	Finally, it follows from \eqref{equa3.10}-\eqref{equa3.13} that
	\begin{equation*}
		\|C_{u,\varphi}-C_{v,\psi}\|_e^p\gtrsim \overline{\lim_{|a|\to 1^-}}M_s(\omega_p)(a)+M_s(\sigma_{\beta,p})(a).
	\end{equation*}
	The proof is complete.
\end{proof}

\section{Proof of Theorem 1.2}

In this section, we aim to prove Theorem \ref{theorem1.2}, which characterizes the Schatten class difference of two weighted composition operators.

\begin{proof}[{\bf Proof of Theorem 1.2}]
	{\it Sufficiency:} Fix $s\in (0,1)$. Assume $M_s(\omega_2)+M_s(\sigma_{\beta,2})\in L^{\frac{p}{2}}(\mathbb{B}_n,d\lambda)$. For any $f\in A_{\alpha}^2(\mathbb{B}_n)$, we conclude from the proof of Theorem \ref{theorem3.1} that
	\begin{equation}\label{equa4.1}
		\begin{split}
			&\langle (C_{u,\varphi}-C_{v,\psi})^*(C_{u,\varphi}-C_{v,\psi})f,f\rangle=\|(C_{u,\varphi}-C_{v,\psi})f\|_{2}^2\\
			&\quad \lesssim \int_{\mathbb{B}_n}|f(z)|^2\left[M_s(\omega_2)(z)+M_s(\sigma_{\beta,2})(z)\right]dv_{\alpha}(z).
		\end{split}
	\end{equation}
	Let $d\mu=\left[M_s(\omega_2)+M_s(\sigma_{\beta,2})\right]dv_{\alpha}$. By \eqref{equa4.1}, we know that $$(C_{u,\varphi}-C_{v,\psi})^*(C_{u,\varphi}-C_{v,\psi})\lesssim T_{\mu}.$$ According to Lemma \ref{lemma2.6}, the assumption implies that $T_{\mu}\in S_{\frac{p}{2}}(A_{\alpha}^2)$. It follows that $(C_{u,\varphi}-C_{v,\psi})^*(C_{u,\varphi}-C_{v,\psi})\in S_{\frac{p}{2}}(A_{\alpha}^2)$, and then $C_{u,\varphi}-C_{v,\psi}\in S_p(A_{\alpha}^2)$.
	
	{\it Necessity:} Now we assume $C_{u,\varphi}-C_{v,\psi}\in S_p(A_{\alpha}^2)$ and $p\geq 2$. Then $(C_{u,\varphi}-C_{v,\psi})^*(C_{u,\varphi}-C_{v,\psi})\in S_{\frac{p}{2}}(A_{\alpha}^2)$. It follows from Lemma \ref{lemma2.5} that the functions
	\[z\mapsto \|(C_{u,\varphi}-C_{v,\psi})k_z\|_2^2\] 
	and
	\[z\mapsto \|(C_{u,\varphi}-C_{v,\psi})k_z^{[R]}\|_2^2\]
	belong to $L^{\frac{p}{2}}(\mathbb{B}_n,d\lambda)$.
	
	For any fixed $s\in (\frac{1}{4},\frac{1}{2})$, by \eqref{equa2.3}, we have
	\begin{equation}\label{equa4.2}
		\begin{split}
			&\|(C_{u,\varphi}-C_{v,\psi})k_z\|_2^2\\
			& \quad\geq \int_{\varphi^{-1}(E(z,s))}\left|\frac{u(\zeta)(1-|z|^2)^{\frac{n+1+\alpha}{2}}}{(1-\langle \varphi(\zeta),z\rangle)^{n+1+\alpha}}-\frac{v(\zeta)(1-|z|^2)^{\frac{n+1+\alpha}{2}}}{(1-\langle \psi(\zeta),z\rangle)^{n+1+\alpha}}\right|^2dv_{\alpha}(\zeta)\\
			& \quad\gtrsim\frac{\int_{\varphi^{-1}(E(z,s))}\left|u(\zeta)-v(\zeta)\frac{(1-\langle\varphi(\zeta),z\rangle)^{n+1+\alpha}}{(1-\langle\psi(\zeta),z\rangle)^{n+1+\alpha}}\right|dv_{\alpha}(\zeta)}{(1-|z|^2)^{n+1+\alpha}}
		\end{split}
	\end{equation}
	and
	\begin{equation}\label{equa4.3}
		\begin{split}
			&\quad\|(C_{u,\varphi}-C_{v,\psi})k_z^{[R]}\|_2^2\\
			&\quad\quad\gtrsim \frac{\int_{\varphi^{-1}(E(z,s))}\left|u(\zeta)\langle \varphi(\zeta),z\rangle-v(\zeta)\langle\psi(\zeta),z\rangle\frac{(1-\langle\varphi(\zeta),z\rangle)^{n+2+\alpha}}{(1-\langle\psi(\zeta),z\rangle)^{n+2+\alpha}}\right|^2dv_{\alpha}(\zeta)}{(1-|z|^2)^{n+1+\alpha}}.
		\end{split}
	\end{equation}
	Note that $\left|\frac{1-\langle\varphi(\zeta),z\rangle}{1-\langle \psi(\zeta),z\rangle}\right|\lesssim 1$ when $\varphi(\zeta)\in E(z,s)$ by \eqref{equa2.3}. Multiplying the integrand in \eqref{equa4.2} by $|\langle \psi(\zeta),z\rangle|^2\left|\frac{1-\langle\varphi(\zeta),z\rangle}{1-\langle \psi(\zeta),z\rangle}\right|^2$, and adding it to \eqref{equa4.3}, by the triangle inequality and \eqref{equa2.4}, we obtain
	\begin{equation}\label{equa4.4}
		\begin{split}
			&\|(C_{u,\varphi}-C_{v,\psi})k_z\|_2^2+\|(C_{u,\varphi}-C_{v,\psi})k_z^{[R]}\|_2^2\\
			&\quad \gtrsim \frac{\int_{\varphi^{-1}(E(z,s))}|u(\zeta)|^2\left|\frac{\langle \psi(\zeta)-\varphi(\zeta),z\rangle}{1-\langle \psi(\zeta),\varphi(\zeta)\rangle}\right|^2dv_{\alpha}(\zeta)}{(1-|z|^2)^{n+1+\alpha}}.
		\end{split}
	\end{equation}
	It follows that
	\begin{equation}\label{equa4.5}
		\int_{\mathbb{B}_n}\left(\frac{\int_{\varphi^{-1}(E(z,s))}|u(\zeta)|^2\left|\frac{\langle \psi(\zeta)-\varphi(\zeta),z\rangle}{1-\langle \psi(\zeta),\varphi(\zeta)\rangle}\right|^2dv_{\alpha}(\zeta)}{(1-|z|^2)^{n+1+\alpha}}\right)^{\frac{p}{2}}d\lambda(z)<\infty.
	\end{equation}
	
	Let 
	$$\eta^j=z+s\sqrt{1-|z|^2}z^j,\quad j=2,\cdots,n.$$
	Then $d(z,\eta^j)<s$ and $E(z,s)\subset E(\eta^j,2s)$. Moreover, the map $z\mapsto \eta^j$ is a bijection on $\mathbb{B}_n$ and
	\[dv(\eta^j)\simeq dv(z).\]
	It follows from \eqref{equa2.3} that
	\begin{equation*}
		\begin{split}
			&\int_{\mathbb{B}_n}\left(\|(C_{u,\varphi}-C_{v,\psi})k_{\eta^j}\|_2^2+\|(C_{u,\varphi}-C_{v,\psi})k_{\eta^j}^{[R]}\|_2^2\right)^{\frac{p}{2}}d\lambda(z)\\
			&\quad \lesssim \int_{\mathbb{B}_n}\left(\|(C_{u,\varphi}-C_{v,\psi})k_{\eta^j}\|_2^2+\|(C_{u,\varphi}-C_{v,\psi})k_{\eta^j}^{[R]}\|_2^2\right)^{\frac{p}{2}}d\lambda(\eta^j)<\infty.
		\end{split}
	\end{equation*}
	Through a similar argument as above and by \eqref{equa2.3}, we obtain
	\begin{equation}\label{equa4.6}
		\begin{split}
			&\|(C_{u,\varphi}-C_{v,\psi})k_{\eta^j}\|_2^2+\|(C_{u,\varphi}-C_{v,\psi})k_{\eta^j}^{[R]}\|_2^2\\
			&\quad \gtrsim \frac{\int_{\varphi^{-1}(E(z,2s))}|u(\zeta)|^2\left|\frac{\langle \psi(\zeta)-\varphi(\zeta),\eta^j\rangle}{1-\langle \psi(\zeta),\varphi(\zeta)\rangle}\right|^2dv_{\alpha}(\zeta)}{(1-|\eta^j|^2)^{n+1+\alpha}}\\
			&\quad\gtrsim \frac{\int_{\varphi^{-1}(E(z,s))}|u(\zeta)|^2\left|\frac{\langle \psi(\zeta)-\varphi(\zeta),\eta^j\rangle}{1-\langle \psi(\zeta),\varphi(\zeta)\rangle}\right|^2dv_{\alpha}(\zeta)}{(1-|z|^2)^{n+1+\alpha}}.
		\end{split}
	\end{equation}
	We combine \eqref{equa4.4} and \eqref{equa4.6} to obtain
	\begin{equation*}
		\begin{split}
			&\|(C_{u,\varphi}-C_{v,\psi})k_{\eta^j}\|_2^2+\|(C_{u,\varphi}-C_{v,\psi})k_{\eta^j}^{[R]}\|_2^2\\
			&\qquad\qquad\qquad\qquad\quad+\|(C_{u,\varphi}-C_{v,\psi})k_{z}\|_2^2+\|(C_{u,\varphi}-C_{v,\psi})k_{z}^{[R]}\|_2^2\\
			&\quad \gtrsim \frac{\int_{\varphi^{-1}(E(z,s))}|u(\zeta)|^2\left|\frac{\langle \psi(\zeta)-\varphi(\zeta),\sqrt{1-|z|^2}z^j\rangle}{1-\langle \psi(\zeta),\varphi(\zeta)\rangle}\right|^2dv_{\alpha}(\zeta)}{(1-|z|^2)^{n+1+\alpha}}.
		\end{split}
	\end{equation*}
	It follows that
	\begin{equation}\label{equa4.7}
		\int_{\mathbb{B}_n}\left(\frac{\int_{\varphi^{-1}(E(z,s))}|u(\zeta)|^2\left|\frac{\langle \psi(\zeta)-\varphi(\zeta),\sqrt{1-|z|^2}z^j\rangle}{1-\langle \psi(\zeta),\varphi(\zeta)\rangle}\right|^2dv_{\alpha}(\zeta)}{(1-|z|^2)^{n+1+\alpha}}\right)^{p/2}d\lambda(z)<\infty.
	\end{equation}
	Combining \eqref{equa4.5} and \eqref{equa4.7}, and by Lemma \ref{lemma2.2}, we get
	\begin{equation*}
		\int_{|z|>t_0}\left(M_s(\omega_{\varphi,u}^2)(z)\right)^{\frac{p}{2}}d\lambda(z)<\infty.
	\end{equation*}
	Similarly,
	\[\int_{|z|>t_0}\left(M_s(\omega_{\psi,v}^2)(z)\right)^{\frac{p}{2}}d\lambda(z)<\infty.\]
	On the other hand,
	\begin{equation*}
		\begin{split}
			&\int_{|z|\leq t_0}\left[M_s(\omega_2)(z)\right]^{\frac{p}{2}}d\lambda(z)\\
			&\quad \lesssim\left(\int_{\mathbb{B}_n}\left(|u(z)|^2+|v(z)|^2\right)|\varphi(z)-\psi(z)|^2dv_{\alpha}(z)\right)^{p/2}\\
			&\quad \lesssim \left(\int_{\mathbb{B}_n}|u(z)-v(z)|^2dv_{\alpha}(z)\right)^{p/2}+\sum_{j=1}^n\left(\int_{\mathbb{B}_n}|u(z)\varphi_j(z)-v(z)\psi_j(z)|^2dv_{\alpha}(z)\right)^{p/2}\\
			&\quad <\infty.
		\end{split}
	\end{equation*}
	Therefore, we get
	\begin{equation}\label{equa4.8}
		\int_{\mathbb{B}_n}\left[M_{s}(\omega_2)(z)\right]^{\frac{p}{2}}d\lambda(z)<\infty.
	\end{equation}
	
	Moreover, when $\varphi(\zeta)\in E(z,s)$, by \eqref{equa2.3}, we have
	\begin{equation*}
		\sup_{t\in [0,1]}\left|\frac{1-(t\langle \varphi(\zeta),z\rangle+(1-t)\langle\psi(\zeta),z\rangle)}{1-\langle \psi(\zeta),z\rangle}\right|\leq 1.
	\end{equation*}
	It follows that
	\begin{equation*}
		\begin{split}
			&|u(\zeta)-v(\zeta)|(1-\rho(\zeta))^{n+1+\alpha}\\
			&\quad \lesssim|u(\zeta)-v(\zeta)|\left|\frac{1-\langle \varphi(\zeta),z\rangle}{1-\langle \psi(\zeta),z\rangle}\right|^{n+1+\alpha}\\
			&\quad \leq \left|u(\zeta)-v(\zeta)\frac{(1-\langle \varphi(\zeta),z\rangle)^{n+1+\alpha}}{(1-\langle \psi(\zeta),z\rangle)^{n+1+\alpha}}\right|+|u(\zeta)|\left|1-\frac{(1-\langle \varphi(\zeta),z\rangle)^{n+1+\alpha}}{(1-\langle \psi(\zeta),z\rangle)^{n+1+\alpha}}\right|\\
			&\quad \lesssim \left|u(\zeta)-v(\zeta)\frac{(1-\langle \varphi(\zeta),z\rangle)^{n+1+\alpha}}{(1-\langle \psi(\zeta),z\rangle)^{n+1+\alpha}}\right|+|u(\zeta)|\left|\frac{\langle \varphi(\zeta)-\psi(\zeta),z\rangle}{1-\langle \varphi(\zeta),\psi(\zeta)\rangle}\right|.
		\end{split}
	\end{equation*}
	So we combine \eqref{equa4.2} and \eqref{equa4.4} to obtain
	\begin{equation}\label{equa4.9}
		\begin{split}
			\infty&>\int_{\mathbb{B}_n}\left(\frac{\int_{\varphi^{-1}(E(z,s))}|u(\zeta)-v(\zeta)|^2(1-\rho(\zeta))^{2(n+1+\alpha)}}{(1-|z|^2)^{n+1+\alpha}}\right)^{\frac{p}{2}}d\lambda(z)\\
			&\gtrsim \int_{\mathbb{B}_n}\left[M_s(\sigma_{\varphi,\beta}^2)(z)\right]^{\frac{p}{2}}d\lambda(z),
		\end{split}
	\end{equation}
	where $\beta\geq 2(n+1+\alpha)$. Similarly,
	\begin{equation}\label{equa4.10}
		\int_{\mathbb{B}_n}\left[M_s(\sigma_{\psi,\beta}^2)(z)\right]^{\frac{p}{2}}d\lambda(z)<\infty.
	\end{equation}
	Combining \eqref{equa4.8}, \eqref{equa4.9} and \eqref{equa4.10}, the proof is complete.
\end{proof}

From the proof of Theorem \ref{theorem1.2}, we conclude the following characterization for Schatten class difference $C_{u,\varphi}-C_{v,\psi}$ via Reproducing Kernel Thesis.

\begin{corollary}\label{corollary4.1}
	Let $\alpha>-1$ and $p\geq 2$. Suppose $u,v\in H(\mathbb{B}_n)$, $\varphi,\psi\in S(\mathbb{B}_n)$ and $C_{u,\varphi}-C_{v,\psi}$ is bounded on $A_{\alpha}^2(\mathbb{B}_n)$. Then $C_{u,\varphi}-C_{v,\psi}\in S_p(A_{\alpha}^2)$ if and only if 
	\begin{equation*}
		\int_{\mathbb{B}_n}\left(\|(C_{u,\varphi}-C_{v,\psi})k_z\|_2^p+\|(C_{u,\varphi}-C_{v,\psi})k_z^{[R]}\|_2^p\right)d\lambda(z)<\infty.
	\end{equation*} 
\end{corollary}

When $p\geq 2$, by taking $v=0$, we conclude from Theorem 1.2 that $C_{u,\varphi}\in S_p(A_{\alpha}^2)$ if and only if
\begin{equation}\label{equa4.11}
	\int_{\mathbb{B}_n}\left[M_s(\mu_{u,\varphi})(z)\right]^{\frac{p}{2}}d\lambda(z)<\infty
\end{equation}
for some (or any) $s\in (0,1)$, where $\mu_{u,\varphi}=(|u|^2dv_{\alpha})\circ\varphi^{-1}$.

\begin{corollary}\label{corollary4.2}
	Let $\alpha>-1$, $p\geq 2$ and $a,b\in \mathbb{C}\backslash\{0\}$. Suppose $aC_{\varphi}+bC_{\psi}$ is bounded on $A_{\alpha}^2(\mathbb{B}_n)$. Neither $C_{\varphi}$ nor $C_{\psi}$ belongs to $S_p(A_{\alpha}^2)$. Then $aC_{\varphi}+bC_{\psi}\in S_p(A_{\alpha}^2)$ if and only if $a+b=0$ and $C_{\varphi}-C_{\psi}\in S_p(A_{\alpha}^2)$.
\end{corollary}

\begin{proof}
	The sufficiency is trivial. We only need to prove the necessity. 
	
	Assume $aC_{\varphi}+bC_{\psi}\in S_p(A_{\alpha}^2)$. Fix $s\in (0,1)$. By Theorem 1.2, we have
	\begin{equation}\label{equa4.12}
		\int_{\mathbb{B}_n}\left(\frac{\int_{\varphi^{-1}(E(z,s))}\rho(\zeta)^2dv_{\alpha}(\zeta)}{(1-|z|^2)^{n+1+\alpha}}\right)^{\frac{p}{2}}d\lambda(z)<\infty
	\end{equation}
	and 
	\begin{equation}\label{equa4.13}
		\int_{\mathbb{B}_n}\left(\frac{\int_{\varphi^{-1}(E(z,s))}(1-\rho(\zeta))^{2(n+1+\alpha)}|a+b|^2dv_{\alpha}(\zeta)}{(1-|z|^2)^{n+1+\alpha}}\right)^{\frac{p}{2}}d\lambda(z)<\infty.
	\end{equation} 
	If $a+b\neq 0$, then combining \eqref{equa4.12} and \eqref{equa4.13}, one could easily obtain
	\[\int_{\mathbb{B}_n}\left(\frac{v_{\alpha}(\varphi^{-1}(E(z,s)))}{(1-|z|^2)^{n+1+\alpha}}\right)^{\frac{p}{2}}d\lambda(z)<\infty.\]
	This shows that $C_{\varphi}\in S_p(A_{\alpha}^2)$ by \eqref{equa4.11}, which is a contradiction. Therefore, $a+b=0$ and $C_{\varphi}-C_{\psi}\in S_p(A_{\alpha}^2)$.
\end{proof}

We end this paper with a characterization for the Hilbert-Schmidtness of the difference $C_{u,\varphi}-C_{v,\psi}$ on $A_{\alpha}^2(\mathbb{B}_n)$, which generalized a result in \cite{AsWz}.

\begin{corollary}\label{corollary4.3}
	Let $\alpha>-1$. Suppose $u,v\in H(\mathbb{B}_n)$, $\varphi,\psi\in S(\mathbb{B}_n)$ and $C_{u,\varphi}-C_{v,\psi}$ is Hilbert-Schmidt if and only if 
	\begin{equation*}
		\int_{\mathbb{B}_n}\rho(z)^2\left(\frac{|u(z)|^2}{(1-|\varphi(z)|^2)^{n+1+\alpha}}+\frac{|v(z)|^2}{(1-|\psi(z)|^2)^{n+1+\alpha}}\right)dv_{\alpha}(z)<\infty
	\end{equation*}
	and
	\begin{equation*}
		\begin{split}
			\int_{\mathbb{B}_n}(1-\rho^2(z))^{n+1+\alpha}|u(z)-v(z)|^2
			\left(\frac{1}{(1-|\varphi(z)|^2)^{n+1+\alpha}}+\frac{1}{(1-|\psi(z)|^2)^{n+1+\alpha}}\right)dv_{\alpha}(z)<\infty.
		\end{split}
	\end{equation*} 
\end{corollary}

\begin{proof}
	By Theorem \ref{theorem1.2}, we know that $C_{u,\varphi}-C_{v,\psi}$ is Hilbert-Schmidt on $A_{\alpha}^2(\mathbb{B}_n)$ if and only if 
	\[\int_{\mathbb{B}_n}\left[M_s(\omega_2)(z)+M_s(\sigma_{\beta,2})(z)\right]d\lambda(z)<\infty\]
	for some $s\in (0,1)$, where $\beta=2(n+1+\alpha)$.
	
	Note that $M_s(\omega_2)=M_s(\omega_{\varphi,u}^2)+M_s(\omega_{\psi.v}^2)$. Using \eqref{equa2.2}, \eqref{equa2.3}, and Fubini's Theorem, we get
	\begin{equation*}
		\begin{split}
			\int_{\mathbb{B}_n}M_s(\omega_{\varphi,u}^2)(z)d\lambda(z)
			& =\int_{\mathbb{B}_n}\frac{1}{(1-|z|^2)^{n+1+\alpha}}\left(\int_{\varphi^{-1}(E(z,s))}|u(\zeta)|^2\rho(\zeta)^2dv_{\alpha}(\zeta)\right)d\lambda(z)\\
			&=\int_{\mathbb{B}_n}\rho(\zeta)^2|u(\zeta)|^2\left(\int_{E(\varphi(\zeta),s)}\frac{1}{(1-|z|^2)^{n+1+\alpha}}d\lambda(z)\right)dv_{\alpha}(\zeta)\\
			& \simeq\int_{\mathbb{B}_n}\rho(\zeta)^2\frac{|u(\zeta)|^2}{(1-|\varphi(\zeta)|^2)^{n+1+\alpha}}dv_{\alpha}(\zeta),
		\end{split}
	\end{equation*}
	and
	\begin{equation*}
		\int_{\mathbb{B}_n}M_s(\omega_{\psi,v}^2)(z)d\lambda(z)\simeq\int_{\mathbb{B}_n}\rho(\zeta)^2\frac{|v(\zeta)|^2}{(1-|\psi(\zeta)|^2)^{n+1+\alpha}}dv_{\alpha}(\zeta).
	\end{equation*}
	Similarly,
	\begin{equation*}
		\begin{split}
			&\int_{\mathbb{B}_n}M_s(\sigma_{\beta,2})(z)d\lambda(z)\\
			&\simeq \int_{\mathbb{B}_n}(1-\rho^2(z))^{\beta}|u(z)-v(z)|^2
			\left(\frac{1}{(1-|\varphi(z)|^2)^{n+1+\alpha}}+\frac{1}{(1-|\psi(z)|^2)^{n+1+\alpha}}\right)dv_{\alpha}(z)\\
			&\leq \int_{\mathbb{B}_n}(1-\rho^2(z))^{n+1+\alpha}|u(z)-v(z)|^2
			\left(\frac{1}{(1-|\varphi(z)|^2)^{n+1+\alpha}}+\frac{1}{(1-|\psi(z)|^2)^{n+1+\alpha}}\right)dv_{\alpha}(z).
		\end{split}
	\end{equation*}
	
	On the other hand, for fixed $r\in (0,1)$, recall that $G_r=\{z\in\mathbb{B}_n:\rho(z)<r\}$. We write
	\begin{equation*}
		\begin{split}
			&\int_{\mathbb{B}_n}(1-\rho^2(z))^{n+1+\alpha}|u(z)-v(z)|^2
			\left(\frac{1}{(1-|\varphi(z)|^2)^{n+1+\alpha}}+\frac{1}{(1-|\psi(z)|^2)^{n+1+\alpha}}\right)dv_{\alpha}(z)\\
			&\quad =\int_{G_r}+\int_{\mathbb{B}_n\backslash G_r}(1-\rho^2(z))^{n+1+\alpha}|u(z)-v(z)|^2\\
			&\qquad\qquad\qquad\qquad\times\left(\frac{1}{(1-|\varphi(z)|^2)^{n+1+\alpha}}+\frac{1}{(1-|\psi(z)|^2)^{n+1+\alpha}}\right)dv_{\alpha}(z).
		\end{split}
	\end{equation*}
	Clearly,
	\begin{equation*}
		\begin{split}
			&\int_{G_r}(1-\rho^2(z))^{n+1+\alpha}|u(z)-v(z)|^2
			\left(\frac{1}{(1-|\varphi(z)|^2)^{n+1+\alpha}}+\frac{1}{(1-|\psi(z)|^2)^{n+1+\alpha}}\right)dv_{\alpha}(z)\\
			&\quad\lesssim \int_{\mathbb{B}_n}M_s(\sigma_{\beta,2})(z)d\lambda(z).
		\end{split}
	\end{equation*}
	And by \eqref{equa2.1}, we have
	\begin{equation*}
		\begin{split}
			&\int_{\mathbb{B}_n\backslash G_r}(1-\rho^2(z))^{n+1+\alpha}|u(z)-v(z)|^2
			\left(\frac{1}{(1-|\varphi(z)|^2)^{n+1+\alpha}}+\frac{1}{(1-|\psi(z)|^2)^{n+1+\alpha}}\right)dv_{\alpha}(z)\\
			&\quad =\int_{\mathbb{B}_n\backslash G_r}|u(z)-v(z)|^2\frac{(1-|\varphi(z)|^2)^{n+1+\alpha}+(1-|\psi(z)|^2)^{n+1+\alpha}}{|1-\langle \varphi(z),\psi(z)\rangle|^{2(n+1+\alpha)}}dv_{\alpha}(z)\\
			&\quad \lesssim\int_{\mathbb{B}_n\backslash G_r}\frac{(|u(z)|^2+|v(z)|^2)\rho(z)^2}{|1-\langle \varphi(z),\psi(z)\rangle|^{n+1+\alpha}}dv_{\alpha}(z)\lesssim \int_{\mathbb{B}_n}M_s(\omega_2)(z)d\lambda(z).
		\end{split}
	\end{equation*}
	Therefore, 
	\begin{equation*}
		\begin{split}
			&\int_{\mathbb{B}_n}\left[M_s(\omega_2)(z)+M_s(\sigma_{\beta,2})(z)\right]d\lambda(z)\\
			&\quad \simeq \int_{\mathbb{B}_n}\rho(z)^2\left(\frac{|u(z)|^2}{(1-|\varphi(z)|^2)^{n+1+\alpha}}+\frac{|v(z)|^2}{(1-|\psi(z)|^2)^{n+1+\alpha}}\right)dv_{\alpha}(z)\\
			&\quad\quad +\int_{\mathbb{B}_n}(1-\rho^2(z))^{n+1+\alpha}|u(z)-v(z)|^2\\
			&\qquad\qquad\qquad\qquad\times\left(\frac{1}{(1-|\varphi(z)|^2)^{n+1+\alpha}}+\frac{1}{(1-|\psi(z)|^2)^{n+1+\alpha}}\right)dv_{\alpha}(z).
		\end{split}
	\end{equation*}
	The proof is now complete.
\end{proof}


\noindent{\bf Acknowledgements}
The authors thank the referees for the comments and suggestions have led to the improvement of the paper. 
~\\

\noindent{\bf Declaration}
The authors declare that they have no conflicts of interest. No data was used for the research described in this article.


\begin{thebibliography}{99}
	
	
	
	
	
	
	
	\bibitem{AsWz} S. Acharyya and Z. Wu,
	\emph{Compact and Hilbert-Schmidt differences of weighted composition operators},
	Integral. Equ. Oper Theory. 88 (2017), 465-482.
	
	\bibitem{Cj} J. Chen,
	\emph{Differences of weighted composition operators on Bergman spaces induced by double weights},
	Bull. Korean Math. Soc. 60 (2023), 1201-1219.
	
	\bibitem{CbCkKhPi} B. Choe, K. Choi, H. Koo and I. Park,
	\emph{Difference of weighted composition operators over the ball},
	Complex Anal. Oper. Theory 18 (2024), Paper No. 33, 34pp.
	
	\bibitem{CbCkKhYj1} B. Choe, K. Choi, H. Koo and J. Yang,
	\emph{Difference of weighted composition operators},
	J. Funct. Anal. 278 (2020), 108401, 38pp.
	
	\bibitem{CbCkKhYj2} B. Choe, K. Choi, H. Koo and J. Yang,
	\emph{Difference of weighted composition operators II},
	Integral Equ. Oper Theory 93 (2021), Paper No. 17, 19pp.
	
	\bibitem{CbHtKh} B. Choe, T. Hosokawa and H. Koo,
	\emph{Hilbert-Schmidt differences of composition operators on the Bergman space},
	Math. Z. 269 (2011), 751-775.
	
	\bibitem{CbKhSw} B. Choe, H. Koo and W. Smith,
	\emph{Difference of composition operators over the half-plane},
	Trans. Amer. Math. Soc. 369 (2017), 3173-3205.
	
	\bibitem{CbKhWm} B. Choe, H. Koo and M. Wang,
	\emph{Compact double differences of composition operators on the Bergman space},
	J. Funct. Anal. 272 (2017), 2273-2307.
	
	\bibitem{CcMb} C. Cowen and B. MacCluer,
	\emph{Composition Operators on Spaces of Analytic Functions},
	Studies in Advanced Mathematics, CRC Press, Boca Raton, FL, 1995.
	
	
	\bibitem{Ff} F. Forelli,
	\emph{The isometries of $H^p$},
	Canad. J. Math. 16 (1964), 721-728.
	
	\bibitem{Kc} C. Kolaski,
	\emph{Isometries of weighted Bergman spaces},
	Canad. J. Math. 34 (1982), 910-915.
	
	\bibitem{KhWm} H. Koo and M. Wang,
	\emph{Joint Carleson measure and the difference of composition operators on $A_{\alpha}^p(\mathbb{B}_n)$},
	J. Math. Anal. Appl. 419 (2014), 1119-1142.
	
	\bibitem{LbRjWf} B. Liu, J. R\"atty\"a and F. Wu,
	\emph{Compact differences of composition operators on Bergman spaces induced by doubling weights},
	J. Geom. Anal. 31 (2021), 12485-12500.
	
	\bibitem{Mj} J. Moorhouse,
	\emph{Compact differences of composition operators},
	J. Funct. Anal. 219 (2005), 70-92
	
	\bibitem{Pi} I. Park,
	\emph{Characterizations for the difference of weighted composition operators over the polydisk},
	J. Math. Anal. Appl. 533 (2024), 128011.
	
	\bibitem{Pj} J. Pau,
	\emph{A remark on Schatten class Toeplitz operators on Bergman spaces},
	Proc. Amer. Math. Soc. 142 (2014), 2763-2768.
	
	\bibitem{Se1} E. Saukko,
	\emph{Difference of composition operators between standard weighted Bergman spaces},
	J. Math. Anal. Appl. 381 (2011), 789-798.
	
	\bibitem{Se2} E. Saukko,
	\emph{An application of atomic decomposition in Bergman spaces to the study of differences of composition operators},
	J. Funct. Anal. 262 (2012), 3872-3890.
	
	\bibitem{Sj} J. Shapiro,
	\emph{Composition Operators and Classical Function Theory},
	Springer-Verlag, New York, 1993.
	
	\bibitem{SjSc}J. Shapiro and  C. Sundberg,
	\emph{Isolation amongst the composition operators},
	Pacific J. Math. 145 (1990), 117-152.
	
	\bibitem{TcYzZz} C. Tong, Z. Yang and Z. Zhou,
	\emph{Toeplitz operators and weighted composition operators on variable exponent Bergman spaces},
	arXiv: 2507.13675 (2025).
	
	\bibitem{YzZz} Z. Yang and Z. Zhou,
	\emph{Atomic decomposition and composition operators on variable exponent Bergman spaces},
	Mediterr. J. Math. 19 (2022), Paper No. 10, 20pp.
	
	\bibitem{ZlZz} L. Zhang and Z. Zhou,
	\emph{Hilbert-Schmidt differences of composition operators between the weighted Bergman spaces on the unit ball},
	Banach J. Math. Anal. 7 (2013), 160-172.
	
	\bibitem{ZrZk} R. Zhao and K. Zhu,
	\emph{Theory of Bergman spaces in the unit ball of $\mathbb{C}^n$},
	Mem. Soc. Math. Fr. (N. s.), No. 115 (2008), vi+103pp.
	
	\bibitem{Zk1} K. Zhu,
	\emph{Spaces of Holomorphic Functions in the Unit Ball},
	Springer-Verlag, New York, 2005.
	
	\bibitem{Zk} K. Zhu,
	\emph{Operator Theory in Function Spaces}, Second Edition.
	Mathematical Surveys and Monographs, American Mathematical Society, 2007.
	
	\bibitem{Zk2} K. Zhu,
	\emph{Schatten class Toeplitz operators on weighted Bergman spaces of the unit ball},
	New York J. Math. 13 (2007), 299-316. 
	
	
	
	
\end{thebibliography}



\end{document}